\documentclass[11pt,a4paper]{amsart}

\usepackage{amssymb,mathtools}
\usepackage{enumitem}
\usepackage{tikz}
\usepackage[hidelinks]{hyperref}
\usepackage[expansion=false,protrusion=false]{microtype}
\usepackage{setspace}
\calclayout

\newtheorem{theorem}{Theorem}
\newtheorem{lemma}[theorem]{Lemma}
\newtheorem{proposition}[theorem]{Proposition}
\newtheorem{corollary}[theorem]{Corollary}
\theoremstyle{remark}
\newtheorem{remark}[theorem]{Remark}
\newtheorem{example}[theorem]{Example}

\DeclareMathOperator*{\argmin}{argmin}
\DeclareMathOperator{\dom}{dom}
\DeclareMathOperator{\Sol}{Sol}
\DeclareMathOperator{\Solw}{Sol^w}
\DeclareMathOperator{\ri}{ri}

\DeclareMathOperator{\inte}{int}
\newcommand{\R}{\mathbb{R}}
\newcommand{\x}{x^*}
\newcommand{\OP}[1]{#1\text{-}\mathrm{OP}}

\begin{document}

\title[Connectedness for convex polynomial multi-objective optimization]{A connectedness
	result for convex polynomial multi-objective optimization without boundedness}

\author{Vu Trung Hieu}
\address{Ho Chi Minh City University of Technology, 268 Ly Thuong Kiet Street, Dien Hong
	Ward, Ho Chi Minh City, Vietnam}
\email{hieuvut@hcmut.edu.vn}


\begin{abstract}
	We prove that the weak Pareto solution set of a multi-objective optimization problem
	with convex polynomial objectives and constraints is connected, assuming neither
	boundedness of the constraint set nor a constraint qualification. Since this set is semi-algebraic, its connectedness implies that it is path-connected whenever it is nonempty. In particular, convex quadratic multi-objective optimization problems have connected weak Pareto solution sets; this answers the weak Pareto part of an open question posed by
	N.~D.~Yen in \emph{Recent Developments in Vector Optimization} (Springer, 2012).
\end{abstract}

\maketitle


\section{Introduction}\label{sec:intro}

For linear data, connectedness of the weak Pareto solution set is classical and needs no
boundedness assumption: the set is a connected union of finitely many faces of the feasible polyhedron; see \cite[Chapter~6]{lucbook} and
\cite[Chapter~7]{ehrgott2005multicriteria}. Beyond the
linear case, the solution set is studied as the image of the solution mapping attached to
a scalarization; the basic tool is the fact that the image of a connected set under an
upper semi-continuous set-valued mapping with connected values is connected; see, e.g.,
\cite[Lemma~2.1]{naccache1978connectedness} and \cite[Theorem~3.1]{warburton1983quasiconcave}. Connected
values are normally secured by generalized convexity of the objectives. Upper semi-continuity usually comes from a coercivity condition or, for convex polynomial
data, from the assumption that the solution mapping of the scalarization is
bounded-valued.

Throughout, $f_1,\dots,f_m$ and $g_1,\dots,g_s$ are convex polynomials on $\R^n$, and
\begin{equation}\label{eq:K}
	K:=\{x\in\R^n:\ g_j(x)\leq 0,\ j=1,\dots,s\}
\end{equation}
is assumed to be nonempty; it is closed and convex, and $s=0$ is allowed, in which case
$K=\R^n$. We write
$f:=(f_1,\dots,f_m)$, and $\OP{m}(K,f)$ denotes the multi-objective optimization problem
of minimizing $f$ over $K$ with respect to the ordering cone $\R^m_+$. A point $\x\in K$
is a \emph{weak Pareto solution} if there is no $x\in K$ with
$f(x)-f(\x)\in-\inte\R^m_+$; the set of such points is $\Solw(K,f)$. We write $K_\infty$
for the \emph{recession cone} of $K$ \cite[Section~8]{rockafellar1970convex},
$\Delta:=\{\xi\in\R^m_+:\sum_i\xi_i=1\}$, $f_{\xi}:=\sum_i\xi_if_i$, and we let
$\Theta:\Delta\rightrightarrows\R^n$, $\Theta(\xi):=\argmin_{x\in K}f_{\xi}(x)$, with $\Lambda:=\dom\Theta=\{\xi\in\Delta:\Theta(\xi)\neq\emptyset\}$.

For such data, Warburton's theorem \cite{warburton1983quasiconcave} gives the connectedness of $\Solw(K,f)$ when $K$ is bounded; then $\Theta$ is bounded-valued and upper semi-continuous. Without assuming $K$ to be bounded, the author and his coauthors proved in \cite[Theorem~6.4]{hieu2025unboundedness} that
$\Solw(K,f)$ is connected provided $\Theta$ has bounded values; there, upper semi-continuity of $\Theta$ at a point is supplied by \cite[Theorem~4.3.3]{bank1983non}, which requires the value at that point to be nonempty and bounded. Boundedness of the values thus enters at a single step of that
proof, but it is a strong requirement: it fails already for the simple convex quadratic
problem
\begin{equation}\label{eq:ex}
	f_1=x_1^2+x_2^2,\qquad f_2=(x_1-1)^2,\qquad K=\R^2,
\end{equation}
where $\Theta(\xi)=\{1\}\times\R$ for $\xi=(0,1)$. This note removes that requirement.

\begin{theorem}\label{thm:main}
	Let $f_1,\dots,f_m$ and $g_1,\dots,g_s$ be arbitrary convex polynomials on $\R^n$ and
	let $K$ be as in \eqref{eq:K}. Then $\Solw(K,f)$ is connected, and path-connected if
	it is nonempty.
\end{theorem}

This answers in the negative the weak Pareto half of
\cite[Question~9.8]{yen2011linear}, which asks, for Pareto and for weak Pareto solutions
alike, whether a convex quadratic problem can have a disconnected solution set.

With upper semi-continuity no longer available, we proceed in two steps. First, we treat
the \emph{bicriteria} case $m=2$, which occupies Section~\ref{sec:bi}; there everything turns on whether the two objectives have a \emph{common} direction of constancy among the nonzero recession directions of $K$. If they have none, the \emph{segment lemma} (Lemma~\ref{lm:segment}) confines the minimizers of $(1-t)f_1+tf_2$ over $K$, for all $t\in(0,1)$, to one fixed compact convex set $K_c$; restricted to that set, the problem has a bounded-valued solution mapping, and \cite[Theorem~6.4]{hieu2025unboundedness} applies. If they have one, we build a \emph{new} bicriteria problem in
\emph{one} dimension less, from whose solution set the original one is recovered as a product
with a half-line or a line; an induction on $n$ closes the argument. Second, we pass
\emph{from $m$ objectives to two}, in Section~\ref{sec:main}: since $\Lambda$ is \emph{convex}
and $\xi\mapsto f_{\xi}$ is affine, fixing $\hat\xi\in\Lambda$ makes every segment $[\hat\xi,\xi]$ with $\xi\in\Lambda$ lie in $\Lambda$ and hence carry a bicriteria problem, so that $\Solw(K,f)$ is the union over $\xi\in\Lambda$ of the sets $\Theta([\hat\xi,\xi]):=\bigcup_{\eta\in[\hat\xi,\xi]}\Theta(\eta)$; each of them is connected by the case $m=2$ and contains $\Theta(\hat\xi)$.
Semi-algebraicity then gives path-connectedness.

\section{Preliminaries}\label{sec:pre}

A set $E\subseteq\R^n$ is \emph{connected} if it is not the union of two disjoint nonempty
relatively open subsets; the empty set is connected by convention. We use three facts from
general topology \cite[\S23]{munkres2000topology}: \textnormal{(T1)} a union of connected sets having a point in common is connected; \textnormal{(T2)} if $X\cong Y$, that is, $X$ and $Y$ are homeomorphic, then $X$ is connected if and only if $Y$ is; and \textnormal{(T3)} for
$F=[0,\infty)$ or $F=\R$, the product $X\times F$ is connected if and only if $X$ is. For set-valued mappings we use the notions of domain
and upper semi-continuity in the sense of \cite[Chapter~1]{aubin1990set}.

Convexity of $K$ and of the $f_i$ gives the scalarization identity
\begin{equation}\label{eq:scalar}
	\Solw(K,f)=\Theta(\Delta)=\bigcup_{\xi\in\Lambda}\Theta(\xi),
\end{equation}
where $\Theta(S):=\bigcup_{\xi\in S}\Theta(\xi)$ for $S\subseteq\Delta$; see, e.g.,
\cite[Chapter~3]{ehrgott2005multicriteria} or \cite[Chapter~4]{ansari2018vector}. No
constraint qualification is needed, the constraints being kept as a set constraint. Each $\Theta(\xi)$ is closed and convex, being empty or the intersection of $K$ with a sublevel set of the continuous convex function $f_\xi$.

Let $p:\R^n\to\R$ be convex, not necessarily polynomial, and let $d\in\R^n$. Since, by convexity of $p$, $t\mapsto(p(x+td)-p(x))/t$ is nondecreasing on $(0,\infty)$, the \emph{recession function}, also called the \emph{asymptotic function} in
\cite{AuslenderTeboulle2003},
\begin{equation}\label{eq:rec}
	p^\infty(d):=\lim_{t\to+\infty}\frac{p(x+td)-p(x)}{t}
	=\sup_{t>0}\frac{p(x+td)-p(x)}{t}\ \in(-\infty,+\infty]
\end{equation}
is well defined and independent of $x$
\cite[Proposition~2.5.2]{AuslenderTeboulle2003}. By \eqref{eq:rec}, $p(x+\sigma d)\leq p(x)+p^\infty(d)\,\sigma$ for all $x$ and all $\sigma>0$; hence
$p^\infty(d)\leq 0$ if and only if $p$ is bounded from above on one (equivalently, every)
ray of direction $d$. As for the opposite sign, $p^\infty(d)\geq 0$ whenever $p$ is bounded from below on one such ray, since otherwise $p(x+\sigma d)\leq p(x)+p^\infty(d)\,\sigma\to-\infty$; and if $p^\infty(d)>0$, then the limit in \eqref{eq:rec} gives $p(x+\sigma d)\to+\infty$ as $\sigma\to+\infty$, for every $x$. For general convex
$p$ the condition $p^\infty(d)\geq 0$ does not imply boundedness from below on the rays
of direction $d$ (take $p(x)=-\log(1+x)$ for $x\geq 0$, $p(x)=-x$ for $x<0$, and $d=1$).
For convex polynomials it does: by Lemma~\ref{lm:lindec} below when
$p^\infty(d)<+\infty$, and, when $p^\infty(d)=+\infty$, because
$p(x+\sigma d)\to+\infty$ as $\sigma\to+\infty$. For convex polynomials, therefore, $p^\infty(d)\geq 0$ is \emph{equivalent} to $p$ being bounded from below on
one (equivalently, every) ray of direction $d$.

\begin{lemma}[affine along $d$]\label{lm:lindec}
	Let $p$ be a convex polynomial on $\R^n$ and $d\in\R^n$ with $p^\infty(d)<+\infty$.
	Then $p(x+\sigma d)=p(x)+p^\infty(d)\,\sigma$ for all $x\in\R^n$ and $\sigma\in\R$.
\end{lemma}

\begin{proof}
	Fix $x$ and put $q(\sigma):=p(x+\sigma d)$, a convex polynomial in one variable. If
	$\deg q\geq 2$, convexity forces the degree to be even with positive leading
	coefficient, so $q(\sigma)/\sigma\to+\infty$ and $p^\infty(d)=+\infty$, contrary to the
	hypothesis. Hence $q$ is affine, with slope
	$\lim_{\sigma\to\infty}(q(\sigma)-q(0))/\sigma=p^\infty(d)$ by \eqref{eq:rec}.
\end{proof}

For $p^\infty(d)=0$ this is the \emph{directional constancy} property of convex polynomials
\cite[Proposition~3.2.1]{AuslenderTeboulle2003}. The property fails for general convex functions:
$e^{-x}$ has recession value $0$ in the direction $1$ without being constant.

\begin{corollary}[dichotomy]\label{cor:dicho}
	Let $p$ be a convex polynomial on $\R^n$ and $d\in\R^n$ such that $p$ is bounded from
	below on some ray of direction $d$. Then exactly one of the following holds:
	\textnormal{(a)} $p(x+\sigma d)=p(x)$ for all $x,\sigma$, and we say that $p$ is
	\emph{constant along $d$}; or \textnormal{(b)} $p(x+\sigma d)\to+\infty$ as
	$\sigma\to+\infty$, for every $x$.
\end{corollary}

Corollary~\ref{cor:dicho} is read off from Lemma~\ref{lm:lindec}: the lower bound forces
$p^\infty(d)\geq 0$, and the alternatives correspond to $p^\infty(d)=0$ and
$p^\infty(d)>0$, the case $p^\infty(d)=+\infty$ falling under (b). Its hypothesis cannot be
dropped: for $p(x)=x$ on $\R$ and $d=-1$ neither alternative holds.

\begin{corollary}[constraints]\label{cor:constr}
	For every $d\in K_\infty$ and every $j$ there is $\gamma_j\geq 0$, depending on $d$, with
	$g_j(x+\sigma d)=g_j(x)-\gamma_j\sigma$ for all $x,\sigma$.
\end{corollary}

Corollary~\ref{cor:constr} is Lemma~\ref{lm:lindec} applied to $g_j$: for any $x_0\in K$ the ray $x_0+\R_+d$ lies in $K$ because $d\in K_\infty$, so $g_j\leq 0$ on it and hence $g_j^\infty(d)=:-\gamma_j\leq 0$; here the sign comes from an \emph{upper} bound, so no lower bound on $g_j$
is needed.

The polynomial structure enters a second time through the following attainment
theorem, which requires no constraint qualification.

\begin{theorem}[Frank--Wolfe-type theorem \cite{belousov2002frank}]\label{thm:FW}
	Let $p,g_1,\dots,g_s$ be convex polynomials and $K$ as in \eqref{eq:K}. If $p$ is bounded from below on $K$, then $p$ attains its minimum there.
\end{theorem}

Since, for $\xi\in\Lambda$,
\begin{equation}\label{eq:Thetasys}
	\Theta(\xi)=\Bigl\{x\in\R^n:\ g_j(x)\leq 0\ (j=1,\dots,s),\
	f_{\xi}(x)-\min_Kf_{\xi}\leq 0\Bigr\}
\end{equation}
is again given by a convex polynomial system, Theorem~\ref{thm:FW} will be applied once
more, in Section~\ref{sec:bi}, to a convex polynomial minimized \emph{on} $\Theta(\xi)$. Finally, by Theorem~\ref{thm:FW}, $\Lambda$ is the set of $\xi\in\Delta$ for which $f_\xi$ is bounded from below on $K$, and
\begin{equation}\label{eq:Lconvex}
	\Lambda\ \text{is convex,}
\end{equation}
since $\eta\mapsto f_\eta$ is affine: if $f_{\xi'}\geq c'$ and $f_{\xi''}\geq c''$ on $K$,
then $f_\xi\geq\min\{c',c''\}$ for $\xi\in[\xi',\xi'']$. This is the argument of
\cite[Theorem~6.4]{hieu2025unboundedness}.

\section{The bicriteria case}\label{sec:bi}

\subsection{Setting and the segment lemma}\label{subs:seg}

Throughout this section $u,v$ are convex polynomials bounded from below on $K$; we write
$\OP{2}(K,u,v)$ for the bicriteria problem and, for $t\in[0,1]$ and $T\subseteq[0,1]$,
\begin{equation}\label{eq:WT}
	p_t:=(1-t)u+tv,\qquad W_t:=\argmin_K p_t,\qquad W_T:=\bigcup_{t\in T}W_t .
\end{equation}
Each $p_t$ is bounded from below on $K$, so $W_t\neq\emptyset$ by Theorem~\ref{thm:FW}, and
$W_{[0,1]}=\Solw(K,(u,v))$ by \eqref{eq:scalar}, the weights $(1-t,t)$, $t\in[0,1]$, exhausting $\Delta$. Let
\[
D:=\bigl\{d\in K_\infty:\ u(x+\sigma d)=u(x),\ v(x+\sigma d)=v(x)
\text{ for all } x\in\R^n,\ \sigma\in\R\bigr\}
\]
be the set of directions in $K_\infty$ along which both $u$ and $v$ are constant (Corollary~\ref{cor:dicho}(a)).

\begin{remark}\label{rmk:D}
	$D$ is a nonempty convex cone. Indeed, $0\in D$; if $d\in D$ and $\lambda\geq 0$, then
	$\lambda d\in K_\infty$ because $K_\infty$ is a cone, and $u,v$ are constant along
	$\lambda d$; and if $d,d'\in D$, then $d+d'\in K_\infty$ because $K_\infty$ is
	convex, and $u(x+\sigma(d+d'))=u((x+\sigma d')+\sigma d)=u(x+\sigma d')=u(x)$ for all $x,\sigma$, and likewise for $v$. It need not be a subspace, since $K_\infty$ need not contain $-d$: for
	$K=\{x_1\geq 0\}\subseteq\R^2$ and $u=v=x_2^2$ one gets $D=\{(d_1,0):d_1\geq 0\}$,
	a ray. For $d\in D\setminus\{0\}$, the threshold function $\tau$ of Section~\ref{subs:deg} appears exactly when $-d\notin K_\infty$: both conditions are equivalent to $\gamma_j>0$ for some $j$ in Corollary~\ref{cor:constr}, that is, to some constraint $g_j$ being strictly decreasing along $d$.
\end{remark}

\begin{lemma}[segment lemma]\label{lm:segment}
	\textnormal{(i)} $W_0=\argmin_Ku$ and $W_1=\argmin_Kv$ are nonempty; $v$ attains its minimum over $W_0$ and $u$ attains its minimum over $W_1$, say at
	\begin{equation}\label{eq:zstar}
		z^*\in\argmin\{v(z):\ z\in W_0\},\qquad w^*\in\argmin\{u(w):\ w\in W_1\} .
	\end{equation}
	\textnormal{(ii)} For every $t\in(0,1)$ and every $x_t\in W_t$,
	$v(x_t)\leq v(z^*)$ and $u(x_t)\leq u(w^*)$.
\end{lemma}

\begin{proof}
	(i) $W_0\neq\emptyset$ by Theorem~\ref{thm:FW}, as $u$ is bounded from below on $K$. By \eqref{eq:Thetasys}, applied to $\OP{2}(K,u,v)$ with $\xi=(1,0)$, $W_0$ is again given by a convex polynomial system, and $v$ is bounded from below on $K\supseteq W_0$; so
	Theorem~\ref{thm:FW} applies once more, now to $v$ on $W_0$, and the minimum in
	\eqref{eq:zstar} is attained. Interchanging $u$ and $v$ handles $W_1$.

	(ii) Fix $t\in(0,1)$ and $x_t\in W_t$. As $z^*\in W_0\subseteq K$ and $x_t$ minimizes
	$p_t$ over $K$, we have $p_t(x_t)\leq p_t(z^*)$, that is
	\[
	(1-t)\bigl[u(x_t)-u(z^*)\bigr]+t\bigl[v(x_t)-v(z^*)\bigr]\leq 0 .
	\]
	Since $u(z^*)=\min_Ku$ and $x_t\in K$, the first bracket is nonnegative, and so is
	$1-t$; hence $t[v(x_t)-v(z^*)]\leq 0$, and dividing by $t>0$ yields
	$v(x_t)\leq v(z^*)$. Interchanging $u$ with $v$, $t$ with $1-t$, and $z^*$ with $w^*$ gives $u(x_t)\leq u(w^*)$. The argument needs $t>0$ and $1-t>0$, which is why (ii) is stated for the open interval only.
\end{proof}

\begin{remark}\label{rmk:segment}
	The point $z^*$ is a \emph{hierarchical} minimizer: it minimizes $v$ among the
	minimizers of $u$. Such points are the viscosity solutions of Attouch
	\cite[Theorem~4.4]{attouch1996viscosity}, and for $v=\frac12\|\cdot\|^2$ and $K=\R^n$,
	$z^*$ is the minimum-norm element of $W_0$ and the first inequality in Lemma~\ref{lm:segment}(ii) is the classical Tikhonov estimate
	$\|x_t\|\leq\|z^*\|$.
\end{remark}

\subsection{The nondegenerate case \texorpdfstring{$D=\{0\}$}{D=0}}\label{subs:nondeg}

When $D=\{0\}$, everything rests on a single compactness property. For $c\in\R$ put
\[
K_c:=\{x\in K:\ u(x)\leq c,\ v(x)\leq c\}.
\]

\begin{lemma}\label{lm:Kc}
	Assume $D=\{0\}$. Then $K_c$ is compact for every $c\in\R$.
\end{lemma}

\begin{proof}
	$K_c$ is closed and convex, being an intersection of $K$ with two sublevel sets of
	continuous convex functions. If $K_c=\emptyset$, it is compact; otherwise, were it
	unbounded, its recession cone would be nontrivial
	\cite[Theorem~8.4]{rockafellar1970convex}, so $x_0+\sigma d\in K_c$ for some
	$x_0\in K_c$, some $d\neq 0$ and all $\sigma\geq 0$. Along that ray $u,v\leq c$ by
	definition of $K_c$, while $u,v$ are bounded from below on $K_c\subseteq K$; so
	Corollary~\ref{cor:dicho} applies to each, and alternative (b) there is excluded by the
	upper bound. Hence both $u$ and $v$ are constant along $d$. As the ray lies in $K$, we
	also have $d\in K_\infty$ \cite[Theorem~8.3]{rockafellar1970convex}, whence
	$d\in D=\{0\}$, contradicting $d\neq0$.
\end{proof}

\begin{proposition}\label{prop:nondeg}
	Assume $D=\{0\}$. Then $W_{[0,1]}$ is connected.
\end{proposition}

\begin{proof}
	Put $c:=\max\{u(w^*),v(z^*)\}$. By Lemma~\ref{lm:Kc} the set $K_c$ is compact, and it
	is given by the convex polynomial system $g_j\leq 0$ $(j=1,\dots,s)$, $u-c\leq 0$,
	$v-c\leq 0$. By Lemma~\ref{lm:segment}(ii), $v(x_t)\leq v(z^*)\leq c$ and
	$u(x_t)\leq u(w^*)\leq c$ for every $t\in(0,1)$ and every $x_t\in W_t$, so
	$W_{(0,1)}\subseteq K_c$; moreover $z^*,w^*\in K_c$, because $u(z^*)=\min_Ku\leq u(w^*)$ and $v(w^*)=\min_Kv\leq v(z^*)$.
	Consider the \emph{subproblem} $\OP{2}(K_c,u,v)$ and put $\widetilde W_t:=\argmin_{K_c}p_t$.
	Since $K_c\subseteq K$ contains a point of every $W_t$, namely $z^*$ if $t=0$,
	every $x_t\in W_t$ if $t\in(0,1)$, and $w^*$ if $t=1$, we have $\min_{K_c}p_t=\min_Kp_t$ and $\widetilde W_t=W_t\cap K_c$ for all
	$t\in[0,1]$; in particular $\widetilde W_t=W_t$ for $t\in(0,1)$. As $K_c$ is compact,
	the solution mapping of the scalarization of $\OP{2}(K_c,u,v)$ has bounded values, so
	$\widetilde W_{[0,1]}=\Solw(K_c,(u,v))$ is connected by
	\cite[Theorem~6.4]{hieu2025unboundedness}, or already by Warburton's theorem
	\cite{warburton1983quasiconcave}. Finally, $W_{[0,1]}=\widetilde W_{[0,1]}\cup W_0\cup W_1$, because $\widetilde W_t=W_t$ for $t\in(0,1)$ and $\widetilde W_t\subseteq W_t$ for $t\in\{0,1\}$; the sets $W_0$ and $W_1$ are convex, hence connected, and they meet $\widetilde W_{[0,1]}$ at $z^*$ and at $w^*$, respectively; so $W_{[0,1]}$ is connected by (T1), applied twice.
\end{proof}

\subsection{The degenerate case \texorpdfstring{$D\neq\{0\}$}{D nontrivial}}\label{subs:deg}

When $D\neq\{0\}$, the problem $\OP{2}(K,u,v)$ is not treated directly. We build instead a \emph{new} problem $\OP{2}(K',u',v')$ of the same type on a hyperplane $H\cong\R^{n-1}$ orthogonal to some $d\in D\setminus\{0\}$; its constraint set $K'=\pi_H(K)$ is the orthogonal projection of $K$ onto $H$ and need not be contained in $K$. Along a direction $d\in D$ the objectives are constant, and so are the
constraints $g_j$ with $\gamma_j=0$ in Corollary~\ref{cor:constr}, whereas those with
$\gamma_j>0$ decrease strictly; Proposition~\ref{prop:split} keeps the constraints that are constant along $d$, which define a relaxation $K_J$ of $K$, and encodes the strictly decreasing ones in a threshold function $\tau$.
That the projection $K'$ is described by the constraints constant along $d$ is the content of \cite[Theorem~1]{belousov2002frank}, and the induction of Proposition~\ref{prop:bi}
follows \cite[Theorem~3]{belousov2002frank}; what is new is part (ii) of Proposition~\ref{prop:split}, the threshold $\tau$ recording \emph{how far} a point of $K'$ must be moved along $d$ to enter $K$, which
turns the reduction into the homeomorphism of Proposition~\ref{prop:deg}.

\begin{proposition}[splitting $K$ along $d$]\label{prop:split}
	Let $d\in D\setminus\{0\}$, put $H:=d^\perp$, and write $x=\pi_Hx+\pi_d(x)\,d$ with $\pi_Hx\in H$ and $\pi_d(x)\in\R$; for $A\subseteq H$ put
	$A\oplus\R d:=\{z+\sigma d:\ z\in A,\ \sigma\in\R\}$. With the $\gamma_j\geq 0$ of
	Corollary~\ref{cor:constr}, put
	\[
	J:=\{j:\gamma_j=0\},\qquad J^c:=\{1,\dots,s\}\setminus J,\qquad
	K_J:=\{x: g_j(x)\leq 0,\ j\in J\},
	\]
	where $K_J:=\R^n$ if $J=\emptyset$, and
	$K':=K_J\cap H$. Then:
	\begin{enumerate}[label=\textnormal{(\roman*)},leftmargin=2.1em,itemsep=0pt,topsep=0pt,parsep=0pt]
		\item $K\subseteq K_J$, $K_J$ is invariant under $\R d$, and $K_J=K'\oplus\R d$
		with $K'=\{z\in H: g_j(z)\leq 0,\ j\in J\}\neq\emptyset$, given by a convex
		polynomial system on $H\cong\R^{n-1}$;
		\item if $J^c\neq\emptyset$, then $\tau:=\max_{j\notin J}g_j/\gamma_j$ is continuous
		on $\R^n$ and, for $x\in K_J$ and $\sigma\in\R$,
		$x+\sigma d\in K\iff\sigma\geq\tau(x)$; if $J^c=\emptyset$, then $K=K_J$;
		\item $K_J=K+\R d$ and $K'=\pi_H(K)$;
		\item for every $t\in[0,1]$, $\inf_{K_J}p_t=\inf_Kp_t$ and $W_t=W^J_t\cap K$, where $W^J_t:=\argmin_{K_J}p_t$;
		\item $u':=u|_H$, $v':=v|_H$ are convex polynomials bounded from below on $K'$, so
		$\OP{2}(K',u',v')$ is a problem of the type considered in this section.
	\end{enumerate}
\end{proposition}

\begin{proof}
	By Corollary~\ref{cor:constr}, $g_j(x+\sigma d)=g_j(x)-\gamma_j\sigma$ with
	$\gamma_j\geq 0$; so the constraints split into those constant along $d$, indexed by
	$J$, and those strictly decreasing along $d$, indexed by $J^c$.

	(i) $K\subseteq K_J$ because $K_J$ retains only part of the constraints. For $j\in J$
	we have $g_j(x+\sigma d)=g_j(x)$ for \emph{all} $\sigma\in\R$, so $K_J$ is invariant under $\R d$. Hence
	$x\in K_J\iff\pi_Hx\in K_J\cap H=K'$, whence $K_J=K'\oplus\R d$. Since $K\subseteq K_J$, projecting a point of $K$ onto $H$ gives a point of $K_J\cap H=K'$, so $K'\neq\emptyset$; the description of $K'$ is just its definition.

	(ii) Let $J^c\neq\emptyset$. As $\gamma_j>0$ for $j\notin J$, $\tau$ is well defined and continuous, being a maximum of finitely many polynomials divided by positive constants. Let $x\in K_J$ and $\sigma\in\R$. For $j\in J$ we have
	$g_j(x+\sigma d)=g_j(x)\leq 0$ because $x\in K_J$, while for $j\notin J$, $g_j(x+\sigma d)=g_j(x)-\gamma_j\sigma\leq 0\iff\sigma\geq g_j(x)/\gamma_j$, because $\gamma_j>0$. Thus $x+\sigma d\in K$ if and only if $\sigma\geq g_j(x)/\gamma_j$ for all $j\notin J$, that is, $\sigma\geq\tau(x)$. If $J^c=\emptyset$, then $K_J=K$ by definition.

	(iii) $K+\R d\subseteq K_J$ by (i). Conversely, for $x\in K_J$ part (ii) provides $\sigma$ with $x+\sigma d\in K$ (any $\sigma$ if $J^c=\emptyset$), so $x\in K+\R d$. Then $\pi_H(K)\subseteq\pi_H(K_J)=K'$ by (i), and for $z\in K'\subseteq K_J$ part (ii) applied to $z$ provides $\sigma$ with $z+\sigma d\in K$, whence $z=\pi_H(z+\sigma d)\in\pi_H(K)$.

	(iv) As $d\in D$, every $p_t$ is constant along $d$. From $K\subseteq K_J$ we get
	$\inf_{K_J}p_t\leq\inf_Kp_t$. Conversely, for $x\in K_J$ part (ii) provides $\sigma$
	with $x+\sigma d\in K$, and then $p_t(x)=p_t(x+\sigma d)\geq\inf_Kp_t$. The two infima
	being equal, a point of $K$ minimizes $p_t$ over $K$ exactly when it minimizes $p_t$
	over $K_J$, i.e., $W_t=W^J_t\cap K$.

	(v) By (iii) and constancy along $d$, $u$ and $v$ are bounded from below on
	$K_J\supseteq K'$ by the same constants as on $K$; restrictions of convex polynomials
	to $H$ are convex polynomials.
\end{proof}

\begin{remark}\label{rmk:relax}
	By Proposition~\ref{prop:split}(iii), passing from $K$ to $K_J=K+\R d$ is a relaxation, and Proposition~\ref{prop:split}(ii) records where each line $x+\R d$, $x\in K_J$, meets $K$: in the half-line $\{x+\sigma d:\sigma\geq\tau(x)\}$ when $J^c\neq\emptyset$, and in the whole line when $J^c=\emptyset$. Cutting $K$ by $H$ directly, that is, replacing $K'$ by $K\cap H$, would lose solutions, since $K\cap H$ may be a proper subset of $K'=\pi_H(K)$, as in Example~\ref{ex:deg}, or even empty, as for the translate of $K$ there by $(1,0)$.
\end{remark}

\begin{proposition}[reduction of the dimension]\label{prop:deg}
	In the setting of Proposition~\ref{prop:split}, with $W'_{[0,1]}$ the set \eqref{eq:WT}
	attached to $\OP{2}(K',u',v')$,
	\begin{equation}\label{eq:Wdeg}
		W_{[0,1]}\cong
		\begin{cases}
			W'_{[0,1]}\times[0,\infty) & \text{if } J^c\neq\emptyset,\\
			W'_{[0,1]}\times\R & \text{if } J^c=\emptyset.
		\end{cases}
	\end{equation}
\end{proposition}

\begin{proof}
	Each $W^J_t$ is invariant under $\R d$, since $K_J$ is by
	Proposition~\ref{prop:split}(i) and $p_t$ is constant along $d$; hence so is
	$W^J_{[0,1]}:=\bigcup_tW^J_t$, and therefore
	$W^J_{[0,1]}=(W^J_{[0,1]}\cap H)\oplus\R d$, as in the proof of
	Proposition~\ref{prop:split}(i).

	We identify the slice $W^J_{[0,1]}\cap H$. As $u'=u|_H$ and $v'=v|_H$, the scalarization of
	$\OP{2}(K',u',v')$ at $t$ is $p_t|_H$, with solution set $W'_t:=\argmin_{K'}p_t$, and
	we claim that
	$
	W^J_t\cap H=W'_t .$
	Indeed, if $x\in W^J_t\cap H$, then $x\in K_J\cap H=K'$ and $p_t(x)=\inf_{K_J}p_t\leq\inf_{K'}p_t$ because $K'\subseteq K_J$; so $x\in W'_t$.
	Conversely, let $z\in W'_t$ and let $x\in K_J$ be arbitrary. Then $\pi_Hx\in K'$ by
	Proposition~\ref{prop:split}(i), and $x-\pi_Hx\in\R d$, so $p_t(x)=p_t(\pi_Hx)$ because
	$p_t$ is constant along $d$, and $p_t(\pi_Hx)\geq p_t(z)$ because $z$ minimizes $p_t$
	over $K'$. So $p_t\geq p_t(z)$ on $K_J$, and $z\in K'=K_J\cap H$; hence
	$z\in W^J_t\cap H$. Taking unions over $t$ gives $W^J_{[0,1]}\cap H=W'_{[0,1]}$, so
	\begin{equation}\label{eq:cylinder}
		W^J_{[0,1]}=W'_{[0,1]}\oplus\R d .
	\end{equation}
If $J^c=\emptyset$, then $K=K_J$ by Proposition~\ref{prop:split}(ii), so $W_{[0,1]}=W^J_{[0,1]}=W'_{[0,1]}\oplus\R d\cong W'_{[0,1]}\times\R$ by \eqref{eq:cylinder}.
If $J^c\neq\emptyset$, then Proposition~\ref{prop:split}(iv), \eqref{eq:cylinder} and
	Proposition~\ref{prop:split}(ii), applied at $z\in K'\subseteq K_J$, give in turn
	\[
	W_{[0,1]}=W^J_{[0,1]}\cap K=\{z+\sigma d:\ z\in W'_{[0,1]},\ \sigma\geq\tau(z)\}.
	\]
	Hence $\Phi(z,\rho):=z+(\tau(z)+\rho)d$ maps $W'_{[0,1]}\times[0,\infty)$ into
	$W_{[0,1]}$, since $\rho\geq0$; it is onto, by the displayed description with
	$\rho:=\sigma-\tau(z)$; it is injective, by projecting onto $H$ and onto $\R d$ in turn;
	it is continuous since $\tau$ is; and its inverse
	$x\mapsto(\pi_Hx,\pi_d(x)-\tau(\pi_Hx))$ is continuous for the same reason.
\end{proof}

By (T2) and (T3), $W_{[0,1]}$ is connected if and only if $W'_{[0,1]}$ is. Geometrically, for $J^c\neq\emptyset$, the map $\Phi$ sends each half-line $\{z\}\times[0,\infty)$ to $z+\tau(z)d+\R_+d$, so that the starting points $z+\tau(z)d$ trace the graph of $\tau|_{W'_{[0,1]}}$; continuity of $\tau$ is what keeps neighboring
half-lines from being displaced abruptly.

\begin{example}[the degenerate reduction with a nonaffine constraint]\label{ex:deg}
	Let $n=2$,
	\[
	K=\{x\in\R^2:\ x_2^2-x_1\leq 0,\ x_2^2-1\leq 0\}=\{x_1\geq x_2^2,\ |x_2|\leq 1\},
	\]
	$u=x_2^2$ and $v=(x_2-1)^2$. Here $K_\infty=\{(a,0):a\geq 0\}$, and for $d=(1,0)$
	both $u$ and $v$ are constant along $d$, being functions of $x_2$ alone; so $d\in D$ and
	$D\neq\{0\}$. Corollary~\ref{cor:constr} gives $g_1(x+\sigma d)=g_1(x)-\sigma$ and
	$g_2(x+\sigma d)=g_2(x)$, so $\gamma_1=1>0$, $\gamma_2=0$, and with $H=\{x_1=0\}$
	Proposition~\ref{prop:split} yields
	\[
	J=\{2\},\qquad K_J=\{|x_2|\leq 1\},\qquad
	K'=\{0\}\times[-1,1],\qquad \tau(x)=x_2^2-x_1 .
	\]
	The threshold is quadratic, and $K\cap H=\{(0,0)\}$ is a single point whereas
	$K'=\pi_H(K)$ is a whole segment. The reduced problem is to minimize
	$(1-t)s^2+t(s-1)^2$ over $s\in[-1,1]$, with minimizer $s=t$, so $W'_{[0,1]}=\{0\}\times[0,1]$, and the description of $W_{[0,1]}$ in the proof of Proposition~\ref{prop:deg} gives
	\[
	W_{[0,1]}=\{(\sigma,s):\ s\in[0,1],\ \sigma\geq s^2\}\cong[0,1]\times[0,\infty).
	\]
	This
	agrees with the direct computation: $p_t$ does not involve $x_1$, so
	$W_t=\{(x_1,t):x_1\geq t^2\}$. The set $W_{[0,1]}$ is unbounded, and its left boundary is a parabolic arc; yet it is homeomorphic to a straight half-strip; see Figure~\ref{fig:deg}.
\end{example}

\begin{figure}[t]
	\centering
	\scalebox{1.45}{%
	\begin{tikzpicture}[x=1.05cm,y=0.80cm,>=stealth,line join=round]
		\colorlet{cKJ}{black!45}
		\colorlet{cK}{teal!70!black}
		\colorlet{cW}{orange!80!black}
		\colorlet{cT}{violet}
		\colorlet{cH}{blue!60!black}
		\fill[black!10] (-1.35,-1) rectangle (2.95,1);
		\draw[cKJ,thick] (-1.35,1) -- (2.95,1);
		\draw[cKJ,thick] (-1.35,-1) -- (2.95,-1);
		\node[cKJ,scale=0.55,above right] at (-1.33,1.0) {$x_2=1$};
		\node[cKJ,scale=0.55,below right] at (-1.33,-1.0) {$x_2=-1$};
		\fill[teal!18] plot[domain=-1:1,samples=50] ({\x*\x},\x)
			-- (2.95,1) -- (2.95,-1) -- cycle;
		\fill[orange!42] plot[domain=0:1,samples=35] ({\x*\x},\x)
			-- (2.95,1) -- (2.95,0) -- cycle;
		\draw[cK,very thick] plot[domain=-1:1,samples=50] ({\x*\x},\x);
		\draw[->] (-1.5,0) -- (3.15,0) node[right] {\scriptsize $x_1$};
		\draw[cH,thick,->] (0,-1.9) -- (0,1.9) node[above] {\scriptsize $x_2$};
		\node[cH,scale=0.65,left] at (-0.05,1.55) {$H$};
		\node[cH,scale=0.65,left] at (-0.05,-1.55) {$H$};
		\draw[cK,line width=3pt] (0,-1) -- (0,1);
		\draw[cW,line width=1.3pt] (0,0) -- (0,1);
		\draw[very thick,->] (0,0) -- (1,0);
		\node[scale=0.6,below] at (0.68,-0.02) {$d=(1,0)$};
		\fill (0,0) circle (1.6pt);
		\draw[cT,thick,dashed] (0,0.72) -- (0.518,0.72);
		\draw[cT,thick,->] (0.518,0.72) -- (2.25,0.72);
		\fill[cT] (0.518,0.72) circle (1.5pt);
		\node[cT,above,scale=0.55] at (0.26,0.74) {$\tau(z)$};
		\node[cK,scale=0.65] at (2.15,-0.55) {$K$};
		\node[cW,scale=0.65] at (2.20,0.33) {$W_{[0,1]}$};
		\node[cKJ,scale=0.65] at (-0.95,-0.55) {$K_J$};
		\node[cK,scale=0.65,left] at (-0.05,-0.55) {$K'$};
		\node[cW,scale=0.65,left] at (-0.05,0.55) {$W'_{[0,1]}$};
		\node[cK,scale=0.55] at (1.25,-0.83) {$x_1=x_2^2$};
	\end{tikzpicture}}
	\caption{Example~\ref{ex:deg}. The region to the right of the parabola is
	{\color{teal!70!black}$K$}; the arrow from the origin is the direction $d=(1,0)$,
	and the line {\color{blue!60!black}$H=\{x_1=0\}$} is the $x_2$-axis. The gray strip
	is the relaxation {\color{black!45}$K_J=K+\R d$}, and its slice by $H$ is the whole
	segment {\color{teal!70!black}$K'=\pi_H(K)$}, whereas $K\cap H$ is the origin alone. The thin orange segment drawn inside $K'$ is {\color{orange!80!black}$W'_{[0,1]}$}, the upper half of $K'$, and the orange part of $K$ is {\color{orange!80!black}$W_{[0,1]}$}: from each
	$z\in W'_{[0,1]}$ a half-line runs in the direction $d$, starting at the point
	$z+\tau(z)d$ of the parabola ({\color{violet}marked in violet}), at distance
	{\color{violet}$\tau(z)$} from $H$.}
	\label{fig:deg}
\end{figure}
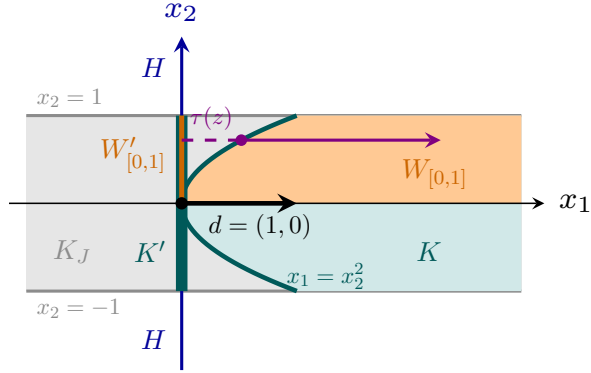

\subsection{The induction}\label{subs:bithm}

\begin{proposition}\label{prop:bi}
	Let $u,v$ be convex polynomials bounded from below on $K$. Then $W_{[0,1]}$ is
	connected.
\end{proposition}

\begin{proof}
	We argue by induction on the dimension $n$ of the ambient space, for all data as in
	Section~\ref{subs:seg}. For $n=0$, $K=\{0\}=W_{[0,1]}$. Let $n\geq 1$ and assume the
	claim in dimension $n-1$. If $D=\{0\}$, apply Proposition~\ref{prop:nondeg}. If
	$D\neq\{0\}$, pick $d\in D\setminus\{0\}$; Proposition~\ref{prop:split}(v) produces a
	problem of the same kind on $H\cong\R^{n-1}$, whose set $W'_{[0,1]}$ is connected by
	the induction hypothesis, and $W_{[0,1]}$ is connected by \eqref{eq:Wdeg} and
	(T2)--(T3).
\end{proof}

\section{Proof of the main theorem}\label{sec:main}

\subsection{From \texorpdfstring{$m$}{m} objectives to two}\label{subs:mto2}

Here the convexity of $\Lambda$ pays off: the segment joining two weights in $\Lambda$ lies in $\Lambda$, so a bicriteria problem can be attached to it.

\begin{proof}[Proof of Theorem~\ref{thm:main}, connectedness]
	For $\Lambda=\emptyset$ the set $\Solw(K,f)$ is empty by \eqref{eq:scalar}, hence
	connected. Assume $\Lambda\neq\emptyset$ and fix once and for all
	$\hat\xi\in\Lambda$. Let $\xi\in\Lambda$ be arbitrary. By \eqref{eq:Lconvex},
	$[\hat\xi,\xi]\subseteq\Lambda$. Put $u:=f_{\hat\xi}$ and $v:=f_{\xi}$, convex
	polynomials bounded from below on $K$ because both weights lie in $\Lambda$; since
	$\eta\mapsto f_\eta$ is affine,
	\[
	(1-t)u+tv=f_{\xi^t},\qquad \xi^t:=(1-t)\hat\xi+t\xi\in\Lambda\quad(t\in[0,1]),
	\]
	so the bicriteria problem $\OP{2}(K,u,v)$ has $W_t=\Theta(\xi^t)$ and
	$W_{[0,1]}=\Theta([\hat\xi,\xi])$. Hence Proposition~\ref{prop:bi}, applied to
	$(u,v)$, shows that $\Theta([\hat\xi,\xi])$ is connected. Moreover, since
	$\hat\xi,\xi\in[\hat\xi,\xi]\subseteq\Lambda$, we have by \eqref{eq:scalar}
	\[
	\Theta(\hat\xi)\cup\Theta(\xi)\subseteq\Theta([\hat\xi,\xi])\subseteq\Theta(\Lambda)
	=\Solw(K,f) .
	\]
	Taking the union over $\xi\in\Lambda$ and using the first inclusion above, we obtain \begin{equation}\label{eq:corewing}
		\Solw(K,f)=\bigcup_{\xi\in\Lambda}\Theta([\hat\xi,\xi]),\quad
		\Theta(\hat\xi)\subseteq\Theta([\hat\xi,\xi])\ (\xi\in\Lambda),
	\end{equation} Thus $\Solw(K,f)$ is a union of connected sets all containing $\Theta(\hat\xi)$, which is nonempty because $\hat\xi\in\Lambda$; hence $\Solw(K,f)$ is connected by (T1).
\end{proof}

\begin{remark}[core and wings]\label{rmk:core}
	Call $\Theta(\hat\xi)$ the \emph{core} and the sets $\Theta([\hat\xi,\xi])$,
	$\xi\in\Lambda$, the \emph{wings} issued from it. The core is convex and nonempty, each
	wing is connected, and by \eqref{eq:corewing} every wing contains the core and the
	solution set is the union of the wings. In particular, any two points of $\Solw(K,f)$ lie in two wings, whose union is connected by (T1): two wings suffice to join any two points, whatever $m$ is.
\end{remark}

\subsection{Path-connectedness}\label{subs:path}

\begin{proof}[Proof of Theorem~\ref{thm:main}, path-connectedness]
	The set $\Solw(K,f)$ is semi-algebraic: it is defined by the first-order formula
	\[
	\x\in K\ \wedge\ \neg\exists x\Bigl(x\in K\wedge\bigwedge\nolimits_{i=1}^m
	f_i(x)<f_i(\x)\Bigr)
	\]
	in the polynomial data, so the Tarski--Seidenberg theorem
	\cite[Chapter~2]{bochnak2013real} applies; this device is used systematically in
	\cite{hieu2020application}. A semi-algebraic set has finitely many connected
	components, each of them semi-algebraic and path-connected
	\cite[Theorem~2.4.5 and Proposition~2.5.13]{bochnak2013real}. Being connected and nonempty, $\Solw(K,f)$ is a single such component and is therefore path-connected.
\end{proof}

\subsection{Scope of the hypotheses}\label{subs:scope}

\begin{remark}[where the polynomial structure is used]\label{rmk:poly}
	The polynomial structure enters at exactly three places: Lemma~\ref{lm:lindec}, used only with
	$p^\infty(d)=0$ in Corollary~\ref{cor:dicho} and with $p^\infty(d)\leq 0$ in
	Corollary~\ref{cor:constr};
	Theorem~\ref{thm:FW}, wherever a solution set is asserted to be nonempty, in
	particular in Lemma~\ref{lm:segment}(i), and for the convexity \eqref{eq:Lconvex} of
	$\Lambda$, which rests on identifying $\Lambda$ with the set of weights $\xi$ for
	which $f_\xi$ is bounded from below on $K$; and semi-algebraicity, for
	path-connectedness only. Everything else -- part (ii) of the segment lemma, the scalarization identity \eqref{eq:scalar}, the connectedness of the weak Pareto
	solution set over a compact convex set \cite{warburton1983quasiconcave}, and the facts (T1)--(T3) -- holds for arbitrary continuous convex data.
\end{remark}

\begin{remark}[weak versus Pareto solutions]\label{rmk:pareto}
	Theorem~\ref{thm:main} concerns \emph{weak} solutions. For the Pareto solution set $\Sol(K,f)$, the set of $\x\in K$ for which no $x\in K$ satisfies $f(x)-f(\x)\in-\R^m_+\setminus\{0\}$, only the inclusions $\bigcup_{\xi\in\Lambda\cap\ri\Delta}\Theta(\xi)\subseteq\Sol(K,f)\subseteq\Solw(K,f)$ are available, with $\ri\Delta$ the relative interior of $\Delta$ \cite[Chapter~3]{ehrgott2005multicriteria}; so the method does not reach $\Sol(K,f)$.
	The obstruction is not peculiar to this method: already for strictly quasi-concave data
	the Pareto solution set may be disconnected, with all its components unbounded
	\cite{hoa2007unbounded}.
\end{remark}

Two questions accordingly remain open. The first is the Pareto half of \cite[Question~9.8]{yen2011linear}; see Remark~\ref{rmk:pareto}. The second is whether $\Solw(K,f)$ is still connected in the general convex case, with $K$ closed convex and the $f_i$ convex and continuous but not necessarily polynomial: no
explicit example of a disconnected weak Pareto solution set is known
\cite[Remark~6.1]{hieu2025unboundedness}, and Theorem~\ref{thm:main} shows that no such
example can have convex polynomial data.

\section*{Acknowledgments}
The author acknowledges Ho Chi Minh City University of Technology (HCMUT), VNU-HCM, for supporting this study. During the preparation of this work, the author used Claude Fable 5 (Anthropic) to assist with the writing and to suggest some elements of the argument. The author reviewed and edited all content and takes full responsibility for the final manuscript.

\bibliographystyle{amsplain-unsrt}
\bibliography{references}

\end{document}